\documentclass{amsart}
\usepackage{pdfpages}
\newtheorem{theorem}{Theorem}[section]

\theoremstyle{definition}
\newtheorem{definition}[theorem]{Definition}

\theoremstyle{remark}

\numberwithin{equation}{section}

\begin{document}

\title{Spherical indicatrices with the modified orthogonal frame}

\author{Mohamd Saleem Lone}
\address{International Centre for Theoretical Sciences, Tata Institute of Fundamental Research, 560089, Bengaluru, India}
\curraddr{International Centre for Theoretical Sciences, Tata Institute of Fundamental Research, 560089, Bengaluru, India}
\email{mohamdsaleem.lone@icts.res.in}

\author{Murat Kemal Karacan}
\address{Usak University, Faculty of Sciences and Arts, Department of
Mathematics,1 Eylul Campus, 64200,Usak-Turkey}
\email{murat.karacan@usak.edu.tr}

\author{Yilmaz Tuncer}
\address{Usak University, Faculty of Sciences and Arts, Department of
Mathematics,1 Eylul Campus, 64200,Usak-TURKEY}
\email{yilmaz.tuncer@usak.edu.tr}

\author{Hasan Es}
\address{Gazi University, Gazi Educational Faculty, Department of
Mathematical Education, 06500 Teknikokullar / Ankara-Turkey}
\email{hasanes@gazi.edu.tr}

\subjclass[2000]{53A04, 53A35}



\keywords{Spherical indicatrix, Darboux indicatrix, modified orthogonal frame.}

\begin{abstract}
In this paper, we study spherical images of the modified orthogonal
vector fields and Darboux vector of a regular curve which lies on the unit
sphere in Euclidean 3-space.
\end{abstract}

\maketitle

\section{ \protect \bigskip Introduction}

Many interesting properties of a space curve $\alpha $ in $E^{3}$ can be investigated by means of the concept of spherical indicatrix of the tangent, principal normal, binormal and Darboux vector to $\alpha $ [7]. Most commonly researchers use the Frenet frame of a curve to characterize the properties of curves. In addition to Frenet frame various frames have been designed. The study of curves with respect to these frames are qualified problems [3]. In [1,2,5-7], the authors have characterized the spherical indicatrices of a
curve in different frames.

\bigskip \ In this paper, we obtain spherical representations of a curve
with respect to the modified orthogonal frame in\ Euclidean 3-space.

\section{Preliminaries}

\bigskip Let $\alpha(s)$ be a $C^{3}$ space curve in Euclidean 3-space $E^3$, parametrized by arc length $s$. We also assume that{ \it its
curvature $\kappa (s)\neq 0$ anywhere}. Then an
orthonormal frame $\left \{ t,n,b\right \} $ exists satisfying the
Frenet-Serret equations%
\begin{equation}
\left[ 
\begin{array}{c}
t^{\prime }(s) \\ 
n^{\prime }(s) \\ 
b^{\prime }(s)%
\end{array}%
\right] =\left[ 
\begin{array}{ccc}
0 & \kappa & 0 \\ 
-\kappa & 0 & \tau \\ 
0 & -\tau & 0%
\end{array}%
\right] \left[ 
\begin{array}{c}
t(s) \\ 
n(s) \\ 
b(s)%
\end{array}%
\right],  \tag{2.1}
\end{equation}%
where $t$ is the unit tangent, $n$ is the unit principal normal, $b$ is the unit binormal, and $\tau (s)$ is the torsion. For a given  $C^{1}$ function $\kappa (s)$  and a continuous function $\tau (s)$, there exists a $C^{3}$ curve $\varphi$ which has an orthonormal frame $\left \{t,n,b\right \} $ satisfying the Frenet-Serret frame (2.1). Moreover, any other curve $\tilde{\varphi}$ satisfying the same conditions, differs from $\varphi$ only by a rigid motion.

Now let $\alpha(s)$ be a general analytic curve which can be reparametrized by its arc length. Assuming that the curvature function has {\it discrete zero points} or $\kappa(s)$ {\it is not identically zero}, we have an orthogonal frame $\left \{ T,N,B\right \} $ defined as follows:%
\begin{equation*}
T=\frac{d\varphi }{ds},\quad  N=\frac{dT}{ds},\quad B=T\times N,
\end{equation*}%
where $T\times N$ is the vector product of $T$ and $N$. The relations
between $\left \{ T,N,B\right \} $ and previous Frenet frame vectors at
non-zero points of $\kappa $ are%
\begin{equation}
T=t,N=\kappa n,B=\kappa b.  \tag{2.2}
\end{equation}%
Thus, we see that $N(s_{0})=B(s_{0})=0$ when $\kappa (s_{0})=0$ and squares of the length
of $N$ and $B$ vary analytically in $s$. From Eq. (2.2), it is easy to calculate

\begin{equation}
\left[ 
\begin{array}{c}
T^{\prime }(s) \\ 
N^{\prime }(s) \\ 
B^{\prime }(s)%
\end{array}%
\right] =\left[ 
\begin{array}{ccc}
0 & 1 & 0 \\ 
-\kappa ^{2} & \frac{\kappa ^{\prime }}{\kappa } & \tau \\ 
0 & -\tau & \frac{\kappa ^{\prime }}{\kappa }%
\end{array}%
\right] \left[ 
\begin{array}{c}
T(s) \\ 
N(s) \\ 
B(s)%
\end{array}%
\right]  \tag{2.3}
\end{equation}%
and 
\begin{equation*}
\tau =\tau (s)=\frac{\det \left( \alpha ^{\prime },\alpha ^{\prime \prime
},\alpha ^{\prime \prime \prime }\right) }{\kappa ^{2}}
\end{equation*}%
is the torsion of $\alpha $. From Frenet-Serret equations, we know that any point, where $\kappa ^{2}=0$ is a removable singularity of $\tau $. Let $\left \langle ,\right \rangle $ be the standard inner product of $E^3$, then $\left \{ T,N,B\right \} $ satisfies:%
\begin{equation}
\left \langle T,T\right \rangle =1,\left \langle N,N\right \rangle =\left
\langle B,B\right \rangle =\kappa ^{2},\left \langle T,N\right \rangle
=\left \langle T,B\right \rangle =\left \langle N,B\right \rangle =0
\tag{2.4}.
\end{equation}%
The orthogonal frame defined in Eq. (2.3) satisfying Eq. (2.4) is called as modified orthogonal frame [8].

\section{Darboux vector with modified orthogonal frame}

Let $\alpha $ be a unit speed curve and let $\left \{ T,N,B\right \} $ be
the modified orthogonal frame at point $\alpha (s)$ along curve $\alpha $ in $%
E^{3}$ and also derivative vectors of the orthogonal frame $\left \{
T,N,B\right \} $ be given as 
\begin{equation}
\left \{ 
\begin{array}{c}
T^{\prime }=N \\ 
N^{\prime }=-\kappa ^{2}T+\dfrac{\kappa ^{\prime }}{\kappa }\  \ N+\tau B \\ 
B^{\prime }=-\tau N+\dfrac{\kappa ^{\prime }}{\kappa }B.%
\end{array}%
\right.  \tag{3.1}
\end{equation}%
Moreover, let $w$ be%
\begin{equation}
w=aT+bN+cB.  \tag{3.2}
\end{equation}%
Then we can find components of vector $w$ such that it satisfies equalities:%
\begin{equation}
\left \{ 
\begin{array}{c}
T^{\prime }=w\times T \\ 
N^{\prime }=w\times N \\ 
B^{\prime }=w\times B.%
\end{array}%
\right.  \tag{3.3}
\end{equation}%
Also one can do the computations below:%
\begin{equation}
\left \{ 
\begin{array}{c}
w\times T=b\left( N\times T\right) +c\left( B\times T\right) =cN-bB \\ 
w\times N=a\left( T\times N\right) +c\left( B\times N\right) =-c\kappa
^{2}T+aB \\ 
w\times B=a\left( T\times B\right) +b\left( N\times B\right) =b\kappa
^{2}T+aN.%
\end{array}%
\right.  \tag{3.4}
\end{equation}%
From Eq.(3.4) we can compute components of the vector $w$ as follows%
\begin{equation}
a=\tau ,\  \  \ b=0,\  \ c=1,\kappa =const.  \tag{3.5}
\end{equation}%
Thus we write vector $w$ and $\left \Vert w\right \Vert $ as%
\begin{equation}
w(s)=\tau (s)T(s)+B(s),\  \  \  \kappa =const.  \tag{3.6}
\end{equation}%
and%
\begin{equation*}
\left \Vert w\right \Vert =\sqrt{\kappa ^{2}+\tau ^{2}},\  \  \  \kappa =const.
\end{equation*}%
$w$ in Eq.(3.6) is called Darboux vector of the curve $\alpha .$ Now we
easily can write the followings equalities%
\begin{equation}
\left \{ 
\begin{array}{c}
T\times T^{^{\prime }}=B \\ 
N\times N^{^{\prime }}=N\times (-\kappa ^{2}T+\dfrac{\kappa ^{\prime }}{%
\kappa }N+\tau B)=\tau \kappa ^{2}T+\kappa ^{2}B=\kappa ^{2}\left( \tau
T+B\right) \\ 
B\times B^{^{\prime }}=B\times (-\tau N+\dfrac{\kappa ^{\prime }}{\kappa }%
B)=\tau \kappa ^{2}T.%
\end{array}%
\right. .  \tag{3.7}
\end{equation}%
Moreover we can write $N\times N^{^{\prime }}$ as%
\begin{equation}
N\times N^{\prime }=\kappa ^{2}w.  \tag{3.8}
\end{equation}%
From Eq.(3.8), we infer the result that $w$ and $N\times N^{\prime }$ are
linear dependent. \ If $\kappa $ equals one then Frenet and the orthogonal
frames coincide for every $s\in I$. Let $\phi $ be angle between $B$ and
darboux vector $w.$ Then we have the following equalities%
\begin{equation}
\left \{ 
\begin{array}{c}
\sin \phi =\frac{\allowbreak \tau }{\sqrt{\kappa ^{2}+\tau ^{2}}}\  \text{\
or }\  \  \allowbreak \tau =\left \Vert w\right \Vert \sin \phi \\ 
\cos \phi =\frac{\allowbreak \kappa }{\sqrt{\kappa ^{2}+\tau ^{2}}}\text{ \
or \  \ }\kappa =\left \Vert w\right \Vert \cos \phi .%
\end{array}%
\right.  \tag{3.9}
\end{equation}
Let $M$ be a hypersurface in $E^{n}$ and the $S$ be the shape operator of $M$ with $D$ acting as covariant derivative. Then, we have the famous  Gauss equation: 
\begin{equation*}
\overline{D}_{X}Y=D_{X}Y+\left \langle S(X),Y\right \rangle N
\end{equation*}
for each $X,Y\in \chi (M)$. 
\section{Arclength of spherical representations of the
curve}

\begin{definition}
Let $\alpha $ be a unit speed regular curve in Euclidean $3-$space with
modified orthogonal frame $T,$ $N$ and $B$. The unit tangent vectors along the
curve $\alpha $ generate a curve $(T)$ on the sphere of radius one about the
origin. The curve $(T)$ is called the spherical indicatrix of $T$ or more
commonly, $(T)$ is called tangent indicatrix of the curve $\alpha $. If $%
\alpha =\alpha (s)$ is a natural representation of the curve $\alpha $, then 
$(T)=T(s)$ will be a representation of $(T)$. Similarly one considers the
principal normal indicatrix $(N)=N(s)$ and binormal indicatrix $(B)=B(s)$
[8].
\end{definition}
\begin{theorem}
Let $\alpha :I\subset R\rightarrow E^{3}$ be a curve and $T(s)$, $N(s)$, $B(s)$ be its tangent, normal and the binormal, then the following hold: 
\begin{itemize}
\item[(1)]\begin{equation*}
D_{T_{_{T}}}T_{_{T}}=-T+\frac{\tau }{\kappa ^{2}}B.  
\end{equation*}
\item[(2)] \begin{equation*}
D_{T_{N}}T_{N}=\frac{\phi ^{\prime }}{\kappa ^{2}\left \Vert w\right \Vert }%
\left( \kappa \sin \phi T+\cos \phi B\right) -\frac{1}{\kappa ^{2}}N. 
\end{equation*}
\item[(3)] \begin{equation*}
D_{T_{B}}T_{B}=\frac{1}{\tau }T-\frac{1}{\kappa ^{2}}B. 
\end{equation*}
\end{itemize}
\end{theorem}
\begin{proof}
Let $T=T(s)$ be the tangent vector field of the curve \ $\alpha :I\subset
R\rightarrow E^{3}.$ The spherical curve $\alpha _{_{T}}=T$ on $S^{2}$ is
called 1$^{st}$ spherical representation of the tangents of $\alpha $. Let $%
s $ be the arclength parameter of\ $\alpha $. If we denote the arclength of
the curve $\alpha _{_{T}}$ by $s_{_{T}}$, then we may write
\begin{equation}
\alpha _{T}(s_{_{T}})=T(s).  \tag{4.1}
\end{equation}%
By differentiating the both side of Eq.(4.1) with respect to $s$ and
using norm, we get%
\begin{equation*}
T_{_{T}}=N\frac{ds_{_{T}}}{ds},
\end{equation*}%
\begin{equation*}
\frac{ds_{_{T}}}{ds}=\kappa
\end{equation*}%
and%
\begin{equation*}
T_{T}=\frac{N}{\kappa }.
\end{equation*}%
If we differentiate with respect to $s$ again, we obtain 
\begin{equation*}
\frac{d}{ds}\left( T_{_{T}}\right) =\frac{d}{ds}\left( \frac{1}{\kappa }%
N\right)
\end{equation*}%
\begin{equation*}
\frac{dT_{T}}{ds_{T}}=\left \{ \left( \frac{1}{\kappa }\right) ^{\prime }N+%
\frac{1}{\kappa }\left( -\kappa ^{2}T+\frac{\kappa ^{\prime }}{\kappa }%
N+\tau B\right) \right \} \frac{ds}{ds_{T}}
\end{equation*}%
\begin{equation*}
\frac{dT_{T}}{ds_{T}}=\left[ \left( \frac{1}{\kappa }\right) ^{\prime
}N-\kappa T+\frac{\kappa ^{\prime }}{\kappa ^{2}}N+\frac{\tau }{\kappa }B%
\right] \frac{1}{\kappa }
\end{equation*}%
\begin{equation}
D_{T_{_{T}}}T_{_{T}}=-T+\frac{\tau }{\kappa ^{2}}B.  \tag{4.2}
\end{equation}
which is part (1) of the theorem.
The spherical curve $\alpha _{_{N}}=N(s)$ on $S^{2}$ is called $2^{nd}$
spherical representation or the spherical representation of $N$ of the
curve $\alpha $. Let $s\in I$ be the arclenght of the curve $\alpha $. If we
denote the arclength of $N$ by $s_{_{N}}$ and take $\kappa =$constant, we
can write%
\begin{equation}
\alpha _{_{N}}(s_{_{N}})=N(s).  \tag{4.3}
\end{equation}%
Differentiating  both side of \ Eq.(4.3) with respect to $s$ , we get%
\begin{equation}
\frac{d\alpha _{_{N}}}{ds_{_{N}}}=\left( -\kappa ^{2}T+\tau B\right) \frac{ds%
}{ds_{_{N}}}.  \tag{4.4}
\end{equation}%
Taking norm of the Eq.(4.4), we obtain%
\begin{equation*}
\frac{ds_{_{N}}}{ds}=\kappa \sqrt{\kappa ^{2}+\tau ^{2}}\text{ or }\frac{%
ds_{_{N}}}{ds}=\kappa \left \Vert w\right \Vert .
\end{equation*}%
Hence we get $T_{_{N}}$ as%
\begin{equation*}
T_{_{N}}=-\frac{\kappa }{\left \Vert w\right \Vert }T+\frac{1}{\kappa }\frac{%
\tau }{\left \Vert w\right \Vert }B
\end{equation*}%
or%
\begin{equation}
\text{ }T_{_{N}}=-\cos \phi T+\frac{1}{\kappa }\sin \phi B\mathrm{.} 
\tag{4.5}
\end{equation}
 If we differentiate Eq.(4.5) with respect to $s$ and by 
Eq.(3.9), we have%
\begin{eqnarray*}
\dfrac{dT_{N}}{ds_{N}} &=&\left[ \phi ^{\prime }\sin \phi T-\cos \phi N+%
\frac{\phi ^{\prime }}{\kappa }\cos \phi B+\frac{\sin \phi }{\kappa }\left(
-\tau N\right) \right] \frac{ds}{ds_{N}} \\
&=&\left[ \phi ^{\prime }\left( \kappa \sin \phi T+\cos \phi B\right)
-\left( \kappa \cos \phi +\tau \sin \phi \right) N\right] \frac{1}{\kappa
^{2}\left \Vert w\right \Vert } \\
&=&\left[ \phi ^{\prime }\left( \kappa \sin \phi T+\cos \phi B\right) -\left
\Vert w\right \Vert N\right] \frac{1}{\kappa ^{2}\left \Vert w\right \Vert }
\end{eqnarray*}%
or 
\begin{equation}
D_{T_{N}}T_{N}=\frac{\phi ^{\prime }}{\kappa ^{2}\left \Vert w\right \Vert }%
\left( \kappa \sin \phi T+\cos \phi B\right) -\frac{1}{\kappa ^{2}}N, 
\tag{4.6}
\end{equation}%
which is (2) of the theorem. 

The spherical curve $B=B(s)$ on $S^{2}$ is called 3$^{rd}$ spherical
representation or the spherical representation of $B$ of the curve $%
\alpha $. Let $s\in I$ be the arclenght of the curve $\alpha $. If we denote
the arclength of $B$ by $s_{B}$ and take $\kappa =$constant, we can write%
\begin{equation}
\alpha _{B}(s_{B})=B(s).  \tag{4.8}
\end{equation}%
Differentiating  Eq.(4.8) with respect to $s,$ we get%
\begin{equation}
T_{B}\frac{ds}{ds_{_{B}}}=-\tau N.  \tag{4.9}
\end{equation}%
Taking the norm of the Eq.(4.9), we obtain

\begin{equation*}
\frac{ds_{B}}{ds}=\kappa \tau .
\end{equation*}%
So we can write 
\begin{equation}
T_{B}=-\frac{N}{\kappa }.  \tag{4.10}
\end{equation}%
Again differentiating(4.10) with respect to $s$, we get%
\begin{equation*}
\frac{dT_{B}}{ds}=\dfrac{d}{ds}\left( -\frac{N}{\kappa }\right) \dfrac{ds}{%
ds_{B}}=-\frac{1}{\kappa }N^{\prime }\allowbreak \frac{1}{\kappa \tau }=-%
\frac{1}{\kappa ^{2}\tau }N^{\prime }\allowbreak =-\frac{1}{\kappa ^{2}\tau }%
\left( -\kappa ^{2}T+\tau B\right)
\end{equation*}%
or 
\begin{equation}
D_{T_{B}}T_{B}=\frac{1}{\tau }T-\frac{1}{\kappa ^{2}}B,  \tag{4.11}
\end{equation}
which is (3) of the theorem.

\end{proof}
\section{The pole curves}
\begin{theorem}
Let $\alpha :I\subset R\rightarrow E^{3}$ be a curve and $C$ be the vector field defined as $C=\frac{w}{\left \Vert w\right \Vert }$, then we have
\begin{equation*}
D_{T_{C}}T_{C}=-\sin \phi \allowbreak T-\dfrac{\cos \phi }{\kappa }B+\frac{%
\left \Vert w\right \Vert }{\phi ^{\prime }\kappa }N. 
\end{equation*}
\end{theorem}
\begin{proof}
Let $s_{C}$ be the arclenght parameter of $(C)$ and the unit vector at
direction of Darboux vector $w$ according to orthogonal frame $E=\left \{
T,N,B\right \} $ be $T_{C}.$ Moreover the geodesic curvature of $(C)$
according to $E^{3}$ be $k_{C}=\left \Vert
D_{_{T_{_{C}}}}T_{_{C}}\right
\Vert .$ Then we can write equation of $(C)$
as 
\begin{equation}
\alpha (s_{C})=C=\frac{w}{\left \Vert w\right \Vert }=\allowbreak \frac{\tau 
}{\left \Vert w\right \Vert }T+\frac{1}{\left \Vert w\right \Vert }B. 
\tag{4.12}
\end{equation}%
From Eq.(3.9), we get%
\begin{equation}
C=\allowbreak \left( \sin \phi \right) T+\left( \frac{\cos \phi }{\kappa }%
\right) B.  \tag{4.13}
\end{equation}

\bigskip Differentiating Eq.(4.13) with respect to $s_{c}$ and by Eq.(2.3),
we get$\allowbreak $%
\begin{equation}
\frac{dC}{ds_{C}}=\frac{\phi ^{\prime }}{\kappa }\left \{ \left( \kappa \cos
\phi T-\sin \phi B\allowbreak \right) +\left( \kappa \sin \phi -\tau \cos
\phi \right) N\right \} \frac{ds}{ds_{_{C}}}.  \tag{4.14}
\end{equation}%
Using Eq.(3.9), we can write (4.14) as%
\begin{equation}
\frac{dC}{ds_{C}}=\frac{\phi ^{\prime }}{\kappa }\left( \kappa \cos \phi
T-\sin \phi B\allowbreak \right) \frac{ds}{ds_{_{C}}}.  \tag{4.15}
\end{equation}%
Since $T_{C}=\frac{dC}{ds_{C}}$, by taking the norm of Eq.(4.15), we get 
\begin{equation*}
\frac{ds_{C}}{ds}=\frac{\phi ^{\prime }}{\kappa }\sqrt{\kappa ^{2}\cos
^{2}\phi +\sin ^{2}\phi \kappa ^{2}\allowbreak }=\phi ^{\prime }
\end{equation*}%
and

\begin{equation}
T_{C}=\cos \phi T-\frac{\sin \phi }{\kappa }B.  \tag{4.16}
\end{equation}

Differentiating Eq.(4.16) with respect to $s$ and using Eq.(2.3) and
Eq.(3.9), we calculate $D_{T_{C}}T_{C}$ as 
\begin{eqnarray*}
D_{T_{C}}T_{C} &=&\left \{ -\phi ^{\prime }\sin \phi \allowbreak T+\cos \phi
N-\left( \frac{\sin \phi }{\kappa }\right) ^{^{\prime }}B-\left( \frac{\sin
\phi }{\kappa }\right) B^{^{\prime }}\right \} \frac{ds}{ds_{_{C}}}\  \  \  \  \
\  \  \  \  \  \  \  \  \  \  \  \  \  \  \  \  \  \  \  \  \  \  \  \  \  \  \\
&=&\left \{ -\phi ^{\prime }\sin \phi \allowbreak T+\cos \phi N-\left( 
\dfrac{\phi ^{\prime }\cos \phi }{\kappa }-\dfrac{\kappa ^{\prime }\sin \phi 
}{\kappa ^{2}}\allowbreak \right) B-\frac{\sin \phi }{\kappa }\left( -\tau N+%
\dfrac{\kappa ^{\prime }}{\kappa }B\right) \right \} \frac{1}{\phi ^{\prime }%
}
\end{eqnarray*}%
\begin{eqnarray*}
D_{T_{C}}T_{C} &=&\left \{ -\phi ^{\prime }\sin \phi \allowbreak T+\cos \phi
N-\dfrac{\phi ^{\prime }\cos \phi }{\kappa }B+\dfrac{\kappa ^{\prime }\sin
\phi \allowbreak }{\kappa ^{2}}B-\allowbreak \frac{\kappa ^{\prime }\sin
\phi }{\kappa ^{2}}B+\frac{\tau \sin \phi }{\kappa }N\right \} \frac{1}{\phi
^{\prime }}\  \\
&=&\left \{ -\sin \phi \allowbreak T+\dfrac{\cos \phi }{\kappa }B+\frac{1}{%
\phi ^{\prime }}\left( \cos \phi N+\frac{\tau \sin \phi }{\kappa }N\right) +%
\frac{1}{\phi ^{\prime }}\left( \dfrac{\kappa ^{\prime }\sin \phi
\allowbreak }{\kappa ^{2}}-\allowbreak \frac{\kappa ^{\prime }\sin \phi }{%
\kappa ^{2}}\right) B\right \}
\end{eqnarray*}%
\begin{eqnarray*}
D_{T_{C}}T_{C} &=&-\sin \phi \allowbreak T-\cos \phi \dfrac{B}{\kappa }+%
\frac{1}{\phi ^{\prime }}\frac{\kappa \cos \phi +\tau \sin \phi }{\kappa }N
\\
&=&-\sin \phi \allowbreak T-\cos \phi \dfrac{B}{\kappa }+\frac{1}{\phi
^{\prime }}\frac{\kappa \frac{\kappa }{\left \Vert w\right \Vert }+\tau 
\frac{\tau }{\left \Vert w\right \Vert }\sin \phi }{\kappa }N
\end{eqnarray*}%
\begin{eqnarray*}
D_{T_{C}}T_{C} &=&-\sin \phi \allowbreak T-\cos \phi \dfrac{B}{\kappa }+%
\frac{1}{\phi ^{\prime }}\frac{\kappa ^{2}+\tau ^{2}}{\left \Vert w\right
\Vert \kappa }N \\
&=&-\sin \phi \allowbreak T-\cos \phi \dfrac{B}{\kappa }+\frac{1}{\phi
^{\prime }}\frac{\left \Vert w\right \Vert ^{2}}{\left \Vert w\right \Vert
\kappa }N\ 
\end{eqnarray*}%
or%
\begin{equation}
D_{T_{C}}T_{C}=-\sin \phi \allowbreak T-\dfrac{\cos \phi }{\kappa }B+\frac{%
\left \Vert w\right \Vert }{\phi ^{\prime }\kappa }N.  \tag{4.17}
\end{equation}
\end{proof}
\section{ Geodesic curvatures of the curves $\left( T\right)
,\left( N\right) $,$\left( B\right) $ and $\left( C\right) $ according to S$%
^{2}$}
\begin{theorem}
Let $\gamma _{T}$, $\gamma _{N}$, $\gamma _{B}$ and $\gamma _{C}$
be geodesic curvatures of $(T)$, $(N),$ $(B)$ and $(C)$ respectively, then the following holds:
\begin{itemize}
\item[(1)] \begin{equation*}
\gamma _{T}=\sqrt{\left( \frac{\phi ^{\prime }}{\kappa \left \Vert w\right \Vert }%
\right) ^{2}+\left( \frac{\kappa ^{2}-1}{\kappa }\right) ^{2}}.
\end{equation*}
\item[(2)] \begin{equation*}
\gamma _{N}=\sqrt{\frac{\left \Vert w\right \Vert ^{2}}{\kappa ^{2}\tau ^{2}}+\kappa^{2}}.
\end{equation*}
\item[(3)]\begin{equation*}
\gamma _{C}=\frac{\left \Vert w\right \Vert }{\phi ^{\prime }}.
\end{equation*}
\end{itemize}
\end{theorem}
\begin{proof}
 Let $\gamma _{T}$, $\gamma _{N}$, $\gamma _{B}$ and $\gamma _{C}$
be geodesic curvatures of $(T)$, $(N),$ $(B)$ and $(C)$ respectively. Using Eqs.(4.2), (4.6), (4.11) and (4.17), we get
\begin{equation*}
\gamma _{T}=\left \Vert D_{T_{T}}T_{T}+\left \langle S(T_{T}),T_{T}\right
\rangle T\right \Vert =\left \Vert D_{T_{T}}T_{T}+1.T\right \Vert =\left
\Vert -T+\frac{\tau }{\kappa ^{2}}B+T\right \Vert =\frac{\tau }{\kappa }%
=\tan \phi ,
\end{equation*}%
\begin{eqnarray*}
\gamma _{_{N}} &=&\left \Vert D_{T_{N}}T_{N}+\left \langle
S(T_{N}),T_{N}\right \rangle N\right \Vert \\
&=&\left \Vert \frac{\phi ^{\prime }}{\kappa ^{2}\left \Vert w\right \Vert }%
\left( \kappa \sin \phi T+\cos \phi B\right) -\frac{1}{\kappa ^{2}}%
N+1.N\right \Vert \\
&=&\left \Vert \frac{\phi ^{\prime }}{\kappa ^{2}\left \Vert w\right \Vert }%
\left( \kappa \sin \phi T+\cos \phi B\right) +\left( \frac{\kappa ^{2}-1}{%
\kappa ^{2}}\right) N\right \Vert \\
&=&\sqrt{\left[ \frac{\phi ^{\prime }}{\kappa ^{2}\left \Vert w\right \Vert }%
\left( \kappa \sin \phi T+\cos \phi B\right) \right] ^{2}+\left( \frac{%
\kappa ^{2}-1}{\kappa ^{2}}\right) ^{2}\kappa ^{2}} \\
\ &=&\sqrt{\left( \frac{\phi ^{\prime }}{\left \Vert w\right \Vert }\right)
^{2}\frac{1}{\kappa ^{4}}\kappa ^{2}+\left( \frac{\kappa ^{2}-1}{\kappa }%
\right) ^{2}} \\
&=&\sqrt{\left( \frac{\phi ^{\prime }}{\kappa \left \Vert w\right \Vert }%
\right) ^{2}+\left( \frac{\kappa ^{2}-1}{\kappa }\right) ^{2},}
\end{eqnarray*}%
\begin{eqnarray*}
\gamma _{B} &=&\left \Vert D_{T_{B}}T_{B}+1.B\right \Vert =\left \Vert \frac{%
1}{\tau }T-\frac{1}{\kappa ^{2}}B+B\right \Vert =\sqrt{\frac{1}{\tau ^{2}}+%
\frac{1}{\kappa ^{2}}+\kappa ^{2}} \\
&=&\sqrt{\frac{\kappa ^{2}+\tau ^{2}}{\kappa ^{2}\tau ^{2}}+\kappa ^{2}} \\
&=&\sqrt{\frac{\left \Vert w\right \Vert ^{2}}{\kappa ^{2}\tau ^{2}}+\kappa
^{2}}
\end{eqnarray*}%
and%
\begin{eqnarray*}
\gamma _{C} &=&\left \Vert D_{T_{C}}T_{C}+\left \langle S(T_{C}),T_{C}\right
\rangle C\right \Vert \\
&=&\left \Vert D_{T_{C}}T_{C}+1.C\right \Vert \\
&=&\left \Vert \left( -\sin \phi \allowbreak T-\dfrac{\cos \phi }{\kappa }B+%
\frac{\left \Vert w\right \Vert }{\phi ^{\prime }\kappa }N\right) +\left(
\sin \phi T+\frac{\cos \phi }{\kappa }B\right) \right \Vert \\
&=&\left \Vert \frac{\left \Vert w\right \Vert }{\phi ^{\prime }\kappa }%
N\right \Vert =\frac{\left \Vert w\right \Vert }{\phi ^{\prime }}.
\end{eqnarray*}
\end{proof}

\begin{theorem}
Let frenet vector fields of a curve $\alpha :I\rightarrow E^{3}$be $T$, $N$, $B$
and also let Darboux vector field of curve $\alpha $ be $w=\tau T+B$ . Let $\alpha _{w},$ $\alpha _{T},$ $\alpha _{N}$ and $\alpha _{B}$ be, respectively,
spherical indicatries of vector fields $w$, $T$, $N$ and $B$ of \ a\ curve%
\textrm{\ }$\alpha $\textrm{\ }in Euclidean $3-$space $E^{3}$.Then

$(i)$ $\alpha _{T}$ is a spherical involute for $\alpha _{\omega},$

$(ii)$ If curve $\alpha $ has the constant curvature, then $\alpha _{B}$ is a
spherical involute for $\alpha _{\omega},$

$(iii)$ If curve $\alpha $ is a helix and $\alpha _{N}$ has the constant curvature, then $\alpha _{N}$ is a spherical involute for $\alpha_{\omega}.$
\end{theorem}
\begin{proof}
Tangents of $(C)$\bigskip \ constant pole and motion pole curves are common.
we already know for spherical indicatries of curve $\alpha :I\rightarrow
E^{3}$ the following equalities%
\begin{equation}
\left \{ \alpha _{T}=T,\alpha _{B}=B,\alpha _{w}=\dfrac{\tau }{\left \Vert
w\right \Vert }T+\dfrac{1}{\left \Vert w\right \Vert }B\ (\kappa
=const.)\right. .  \tag{5.2}
\end{equation}%
Similary, we also know that 
\begin{equation}
\left \{ 
\begin{array}{c}
\alpha _{T}^{\prime }=T_{T}=N(s)\frac{ds}{ds_{T}} \\ 
\alpha _{N}^{\prime }=T_{N}=\left( -\kappa ^{2}T+\tau B\right) \frac{ds}{%
ds_{N}} \\ 
\alpha _{B}^{\prime }=T_{B}=-\tau T\frac{ds}{ds_{B}}, \\ 
\alpha _{C}^{\prime }=T_{C}=\phi ^{\prime }\left( \cos \phi T-\sin \phi \ 
\dfrac{B}{\kappa }\right) \frac{ds}{ds_{C}}.%
\end{array}%
\right.  \tag{5.3}
\end{equation}%
\bigskip

(i) From inner product of unit tangent vector fields of $T_{T}$ and $T_{C}$,
we get%
\begin{equation}
\left \langle T_{T\text{ }},T_{C}\right \rangle =\frac{ds}{ds_{T}}\phi
^{\prime }\left \langle N,\cos \phi T-\sin \phi \  \frac{B}{\kappa }\right
\rangle =0.\  \  \  \  \  \  \  \  \  \  \  \  \  \  \  \  \  \  \  \  \  \  \  \  \  \  \  \  \  
\tag{5.4}
\end{equation}%
So we conclude that $\alpha _{T}$ is a spherical involute of $\alpha _{\omega}$.

(ii) From inner product of unit tangent vector fields of $T_{B}$ and $T_{C}$,
we get 
\begin{eqnarray*}
\left \langle T_{B},T_{C}\right \rangle &=&\left \langle -\tau N\tfrac{ds}{%
ds_{B}},\phi ^{\prime }\left( \cos \phi T-\sin \phi \  \dfrac{B}{\kappa }%
\right) \frac{ds}{ds_{C}}\right \rangle  \notag \\
&=&-\tau \phi ^{\prime }\frac{ds}{ds_{B}}\frac{ds}{ds_{C}}\left \langle
N,\cos \phi T-\sin \phi \  \dfrac{B}{\kappa }\right \rangle =0 .
\end{eqnarray*}%
So $\alpha _{B}$ is a spherical involute of $\alpha _{\omega}$.

(iii) From inner product of unit tangent vector fields of $T_{N}$ and $T_{C}$%
, we get \textrm{\ }%
\begin{eqnarray*}
\left \langle T_{N},T_{C}\right \rangle &=&\left \langle \left( -\kappa
^{2}T+\tau B\right) \frac{ds}{ds_{N}},\phi ^{\prime }\left( \cos \phi T-\sin
\phi \  \dfrac{B}{\kappa }\right) \frac{ds}{ds_{C}}\right \rangle  \notag \\
&=&\phi ^{\prime }\frac{ds}{ds_{N}}\frac{ds}{ds_{C}}\left \langle -\kappa
^{2}T+\tau B,\cos \phi T-\sin \phi \  \dfrac{B}{\kappa }\right \rangle  \notag
\\
&=&-\phi ^{\prime }\frac{ds}{ds_{N}}\frac{ds}{ds_{C}}\kappa \left( \kappa
\cos \phi +\tau \sin \phi \right)  \notag \\
&=&-\kappa \phi ^{\prime }\frac{ds}{ds_{N}}\frac{ds}{ds_{C}}\left[ \tau \cos
\phi \kappa +\sin \phi \right]  \notag \\
\left \langle T_{N},T_{C}\right \rangle &=&\phi ^{\prime }\kappa \left \Vert
w\right \Vert \frac{ds}{ds_{N}}\frac{ds}{ds_{C}}
\end{eqnarray*}

Since $\alpha $ is a helix, $\frac{\tau }{\kappa }=-\cot \phi $ is constant.
We obtain $\phi ^{\prime }$ $=0$. Thus the proof is finished.

\textrm{\  \  \  \  \  \  \  \  \  \  \  \  \  \  \  \  \  \  \  \ }
\end{proof}


\bigskip

\begin{thebibliography}{9}
\bibitem{1} H.H. Hac\i salihoglu, \textit{A new characterization for
inclined curves by the help of spherical representations}, International
Electronic Journal of Geometry, \textbf{2 }(2), (2009), 71-75.

\bibitem{2} I. Arslan and H.H. Hac\i salihoglu, \textit{On the spherical
representatives of a curve}, Int. J. Contemp. Math. Sciences, \textbf{4}
(34), (2009), 1665-1670.

\bibitem{3} I.G.Arslan and S.Kaya Nurkan, \textit{The Relation Among Bishop
Spherical Indicatrix Curves} , International Mathematical Forum, \textbf{6}%
(25), (2011), 1209-1215

\bibitem{4} M.P. Do Carmo, Differential geometry of curves and surfaces,
Prentice Hall, Englewood Cliffs, NJ, 1976.

\bibitem{5} N.Ekmekci,O.Z.Okuyucu and Y.Yayli,\textit{Characterization of
Spherical Helices in Euclidean 3-Space,} An. S t. Univ. Ovidius Constanta, 
\textbf{22 }(2), (2014), 99-108.

\bibitem{6} S.Y\i lmaz, E.Ozy\i lmaz and M. Turgut, \textit{New spherical
indicatrices and their characterizations}, An. St.Univ. Ovidius Constanta, 
\textbf{18 }(2), (2010), 337-354

\bibitem{7} S.K.Chung, \textit{A Study on the Spherical Indicatrix of a
Space Curve in }$E^{3},$Journal of the Korean Society of Mathematical
Education, XX(3), (1982), 23-24.

\bibitem{8} T. Sasai,\textit{The Fundamental Theorem Of Analytic Space
Curves And Apparent Singularities of Fuchsian Differential Equations},Tohoku
Math. Journ. {\bf 36}, (1984), 17-24.

$\allowbreak $
\end{thebibliography}
\end{document}